\documentclass[12pt]{amsart}
\usepackage{amsmath,amsthm,latexsym,amscd,amsbsy,amssymb,url}
\usepackage[all]{xypic}
 \relax

\chardef\bslash=`\\ 

\makeatletter
\def\verbatim{\interlinepenalty\@M \@verbatim
  \leftskip\@totalleftmargin\advance\leftskip2pc
  \frenchspacing\@vobeyspaces \@xverbatim}
\makeatother
\newtheorem{thm}{Theorem}[section]

\newtheorem{lem}[thm]{Lemma}

\newtheorem{ex}[thm]{Example}
\newtheorem{que}[thm]{Question}
\newtheorem{prob}[thm]{Problem}

\begin{document}


\title
{When is $X\times Y$ homeomorphic  to $X\times_l Y$?}
\author{Raushan  Buzyakova}
\email{Raushan\_Buzyakova@yahoo.com}



\begin{abstract}{
We identify a class of linearly ordered topological  spaces $X$ that may satisfy the property that
$X\times X$ is homeomorphic to $X\times_l X$  or can be embedded into a linearly ordered space  with the stated property. We justify the conjectures by partial results.
}
\end{abstract}

\maketitle
\markboth{R. Buzyakova}{When is $X\times Y$ homeomorphic to $X\times_l Y$?}
{ }

\section{Questions}\label{S:questions}
\par\bigskip
In this paper we provide a discussion that justifies our interest in the question of the title. We also identify more specific questions that may lead to  affirmative resolutions. We back up our curiosity by some partial results and examples. The main result of this work is Theorem \ref{thm:main}. To proceed further let us agree on some terminology. A linear order will also be called an order. An order $<$ on $X$ is compatible with the topology of $X$, if the topology  induced by $<$ is equal to the topology of $X$. A linearly ordered topological space (abbreviated as LOTS) is a pair $\langle X, <\rangle$ of a topological space $X$ and a topology-compatible order $<$ on $X$. A topological space $X$ is orderable if its topology can be induced by some order on $X$. When we consider the lexicographical product $X\times_l Y$ of two LOTS $X$ and $Y$, we first take the lexicographical products of the ordered sets $X$ and $Y$ and then induce the topology as determined by the lexicographical order on $X\times_l Y$. For the purpose of readability we will assume an informal style when describing some folklore-type structures or arguments.

The operations of Cartesian product and lexicographical product produce (more often than not) completely different structures. The former results in a visually "more voluminous" structure, while the latter keeps "visual linearity" but introduces "stretches". In rare cases, however, both operations produce the same results from a topological point of view.  For example, $\mathbb Q\times \mathbb Q$ is homeomorphic to $\mathbb Q\times_l\mathbb Q$. Also, $S\times S$ is homeomorphic to $S\times_l S$, where $S=\{\pm 1/n:n=1,2,...\}$. Note that $S$ is homeomorphic to the space $\mathbb N$ of natural numbers. However, $\mathbb N\times \mathbb N$ is not homeomorphic to $\mathbb N\times_l \mathbb N$. Indeed, the former is discrete while the latter has non-isolated points such as $\langle 2,1\rangle , \langle 3,1\rangle$, etc.
Following this discussion, it is not hard to see that given any discrete space $D$, it is possible to find a topology-compatible   order $\prec$ on $D$ such that $D^*=\langle D, \prec\rangle$ is discrete and $D^*\times D^*$  is homeomorphic to $D^*\times_l D^*$. Our discussion prompts the following general problem.

\par\bigskip\noindent
\begin{prob}
What conditions on $X$ guarantee that there exists a topology-compatible order $\prec$ on $X$ such that  $X\times X$ is homeomorphic to $\langle X, \prec\rangle\times_l \langle X, \prec\rangle$?
\end{prob}

\par\bigskip\noindent
Note that homogeneity is not a necessary condition as follows from the following folklore fact.

\par\bigskip\noindent
\begin{ex}\label{ex:convseq} (Folklore)
$(\omega + 1)\times_l (\omega + 1)$ is homeomorphic to $(\omega + 1) \times (\omega + 1)$.
\end{ex}
\begin{proof}
First observe that $Y = [(\omega + 1)\times_l (\omega +1)]\setminus \{\langle \omega, n\rangle: n= 1,2,..\}$
is homeomorphic to $(\omega + 1)\times_l (\omega +1)$. We will, therefore, provide a homeomorphism between $X= (\omega + 1)\times (\omega +1)$ and $Y$. We define our homeomorphism  in three stages as follows:
\begin{description}
	\item[\it Stage 1] For every $n\in \omega$, fix a  bijection  $f_n$ between $\{\langle n, k\rangle \in X: k = n, ..., \omega\}\subset (\omega +1)\times (\omega +1)$ and $\{\langle 2n, m\rangle\in Y: m\in \omega+1\}\subset (\omega +1)\times_l (\omega +1)$. Such a homeomorphism exists since both subspaces are homeomorphic to $\omega +1$.

	\item[\it Stage 2]  For every $n\in \omega$, fix a bijection  $g_n$  between $\{\langle k,n\rangle \in X: k = n+1, ..., \omega\}$ and $\{\langle 2n+1, m\rangle\in Y: m\in \omega+1\}$.

	\item[\it Stage 3] Define the promised homomorphism from $X$ to $Y$ as follows:

$$
f(x) = \left\{
        \begin{array}{ll}
             f_n(x) & x \in  \{\langle n, k\rangle \in X: k = n, ..., \omega\}\\
             g_n(x)& x\in \{\langle k,n\rangle \in X: k = n+1, ..., \omega\} \\
	  \langle \omega, 0\rangle & x=\langle \omega,\omega\rangle
        \end{array}
    \right.
$$
Visually, $f$ maps the $n$-th vertical at or above the diagonal in $(\omega +1)\times (\omega +1)$ onto the $(2n)$-th copy of $(\omega +1)$ in $(\omega + 1)\times_l (\omega +1 )$. Also $f$ maps the $n$-th horizontal under the diagonal in $(\omega +1)\times (\omega +1)$ onto the $(2n+1)$-st copy of $(\omega +1)$ in $(\omega + 1)\times_l (\omega +1 )$. Finally, $f$ maps the upper right corner point of the Cartesian product to $\langle \omega ,0\rangle$ of the lexicographical product, which is the only point that is the limit of a sequence of non-isolated points.
\end{description}
Clearly, $f$ is a bijection. Let us show that $f$ and $f^{-1}$ are continuous.    Since the domains and images of $f_n$ and $g_n$ are clopen in the respective superspaces, it remains to show that $f$ is continuous at $\langle \omega,\omega\rangle$ and $f^{-1}$ is continuous at $\langle \omega,0\rangle$. For this let $U_n=[n,\omega]\times [n,\omega]$. Then $f(U)= \{\langle (a,b)\in Y: a\geq 2n\}$, which is an open neighborhood of $\langle \omega, 0\rangle $ in $Y$. We have $\{U_n\}_n$ is a basis at $\langle \omega,\omega\rangle $ in $X$ and $\{f(U_n)\}_n$ is a basis at $\langle \omega , 0\rangle$ in $Y$. Since $Y$ is bijective, $f^{-1}$ is continuous at $\langle \omega,0\rangle$.  We proved that $(\omega + 1)\times (\omega +1)$ is homeomorphic to $Y$, and therefore, to $(\omega +1)\times_l(\omega +1)$.
\end{proof}

\par\bigskip\noindent
Even though $(\omega + 1)$ is not homogeneous, it is homogeneous at all non-isolated points (since there is only such point). But even this property is not necessary for the two types of products to be homeomorphic. A similar argument can be used to verify the presence of the studied  phenomenon  in the following example.

\par\bigskip\noindent
\begin{ex}\label{ex:seqofsequences}
$X\times_l X$ is homeomorphic to $X \times X$, where $X=(\omega\times_l\omega ) + 1$ .
\end{ex}

\par\bigskip\noindent
The limit points in this example have different natures. The leftmost point cannot be carried by a homeomorphism to any internal limit point. 
We omit the proof of the statement of Example \ref{ex:seqofsequences} since we will prove a more general one later (Lemma \ref{lem:main}).
Following Example \ref{ex:convseq} and the fact that any discrete space has the property under discussion, one may wonder if any linearly ordered space with a single non-isolated point has the property. The following example shows that the answer is negative and opens another direction for our study.

\par\bigskip\noindent
\begin{ex}
Let $X= (\omega +1 )\oplus D$, where $D$ is an $\omega_1$-sized discrete space. Then the following hold:
\begin{enumerate}
	\item $X$ is orderable.
	\item  $X\times X$ is not homogeneous to $\langle X, \prec\rangle \times_l \langle X, \prec\rangle$ for any topology-compatible order $\prec$ on $X$.
\end{enumerate}
\end{ex}
\begin{proof} To see why $X$ is orderable, first observe that we can think of $X$ as the subspace of $\omega_1$ that contains only all  isolated ordinals   of $\omega_1$ and the ordinal $\omega$.  To order $X$, simply reverse the order of every sequence in form $\{\alpha +1, \alpha +2, ...,\}$, for each limit ordinal greater than $\omega$.

To prove part (2), fix an arbitrary topology-compatible order $\prec$ on $X$. 
The space $X_\prec=\langle X,\prec \rangle$ has at least one of extreme points or neither. Let us consider all possibilities.
\begin{description}
	\item[\rm Case ({\it $X_\prec$ has neither minimum nor maximum})] Then $\{x\}\times_l X_\prec$ is clopen in $X_\prec\times_l X_\prec$ for each $x$. Therefore, 
$X_\prec\times_l X_\prec$ is the free sum of $\omega_1$ many topological copies of $X$. Hence, 
$X_\prec \times_l X_\prec$ is not homeomorphic to $X\times X$.
	\item[\rm Case ({\it $ X_\prec$ has minimum but not maximum})]  Assume first that 
$X_\prec$ has a strictly increasing sequence $\{a_n\}_n$ converging to $\omega$. Then any neighborhood of 
$\langle \omega, \min X_\prec\rangle$ contains $\{a_n\}\times_l X_\prec$.  Therefore,
any neighborhood of 
$\langle \omega, \min X_\prec\rangle$ has size $\omega_1$, while no point in $X\times X$ has such populous base neighborhoods.

We now assume that $X_\prec$ has no strictly increasing sequences converging to $\omega$. This and the absence of a maximum imply that $X_\prec \times X_\prec$ does not have a topological copy of $\omega\times_l\omega + 1$. However,  $X\times X$ does, which is  $\langle \omega,\omega\rangle$. In other words, the second derived set of the lexicographical product is empty but $(X\times X)'' = \{\langle\omega,\omega\rangle\}$.
	\item[\rm Case ({\it $\langle X,\prec \rangle$ has maximum but not minimum})] Similar to Case 2.
	\item[\rm Case ({\it $\langle X,\prec \rangle$ has both maximum and minimum})] Similar to the first part of Case 2.
\end{description}
Since we have exhausted all cases, the roof is complete.
\end{proof}

\par\bigskip
It is known (see, for example \cite{Buz}) that given a subspace $X$ of an ordinal, the square of  $X$ is homeomorphic to a subspace of a linearly ordered space if and only if $X$ has no stationary subsets and is character homogeneous at all non-isolated points.  This statement and the preceding discussion lead to the following question.

\par\bigskip\noindent
\begin{que}\label{que:main}
Let $X$ be a subset of an ordinal, character homogeneous at non-isolated points, and have no stationary subsets. Can $X$ be embedded in a linearly ordered space $L$ for which  $L\times L$ and $L\times_l L$ are homeomorphic?
\end{que}

\par\bigskip
Note that even though spaces in examples \ref{ex:convseq} and \ref{ex:seqofsequences} are not homogeneous, each point has a basis of mutually homeomorphic neighborhoods. This observation prompts the following question.

\par\bigskip\noindent
\begin{que}\label{que:main2}
Let $X$ be a subspace of an ordinal and every point of $X$ has a basis of mutually homeomorphic neighborhoods.  Can $X$ be embedded in a linearly ordered space $L$ for which  $L\times L$ and $L\times_l L$ are homeomorphic?
\end{que}

\par\bigskip
In the next section we will justify the discussed questions by proving a statement that generalizes Example \ref{ex:convseq}. Namely, we will show that if $X$ is a subspace of an ordinal and is homogeneous on its derived set $X'$, then $X$ is embeddable in a linearly ordered space $L$ that has homeomorphic  Cartesian and lexicographical products (Theorem \ref{thm:main}). To prove this we will first identify a special class of spaces for which the two types of products are homeomorphic (Lemma \ref{lem:main}). The structure of these spaces is similar to that of the space in Example \ref{ex:seqofsequences}. We, therefore, generalize Example \ref{ex:seqofsequences} too.

 In notations and terminology we will follow \cite{Eng}. If $X$ is a linearly-ordered set,  by $[a,b]_X$ we denote the closed interval in $X$. If it is clear that the interval is considered in $X$ but not in some larger ordered set, we simply write $[a,b]$. The same concerns other types of intervals. By $X'$ we will denote the set of all non-isolated points of $X$, that is, {\it the derived set of $X$}. We also say that
$X$ is {\it homogeneous on its subset} $A$ if  for every $x,y\in A$ there exists a homeomorphism  $f:X\to X$  such that $f(x)=y$ and $f(y)=x$.

\par\bigskip
\section{Partial Results}\label{S:partialresults}

\par\bigskip
 In what follows, by $\mathcal L$ we denote {\it the class of all subspaces of ordinals that are homeomorphic on their derived sets.}

\par\bigskip
To prove our main statement (Theorem \ref {thm:main}), we start with two  technical lemmas about the key properties of the members of $\mathcal L$ that will be used in further arguments.
\par\bigskip\noindent
\begin{lem}\label{lem:homogeneousnbhds}
Let $X\in \mathcal L$. Then, for any $x\in X'$ there exists $\alpha_x<x$ such that $x$ is the single non-isolated point of $[\alpha_x,x]_X$.
\end{lem}
\begin{proof} By homogeneity of $X$ on $X'$ it suffices to show that the conclusion holds for some element of $X'$.
We may assume that $X'$ is not empty. Then
$z=\min X'$ is defined. Then $\alpha_z = \min X$ is as desired for $z$. 
\end{proof}

\par\bigskip\noindent
\begin{lem}\label{lem:freesum}
Let $X\in \mathcal L$.  Then, $X$ can be written as 
$\left ( \oplus_{x\in X'}I_x\right ) \oplus D$ so that the following hold:
\begin{enumerate}
	\item  $D$ is  clopen and discrete,
	\item  $I_x$ and $I_y$  are  homeomorphic for any $x,y\in X'$,
	\item  $x$ is the only non-isolated point of $I_x$ for each $x\in X'$.
\end{enumerate}
\end{lem}
\begin{proof}
For any $x\in X'$, let $\alpha_x$ be as in Lemma \ref{lem:homogeneousnbhds}. We can find $\beta_x$ between $\alpha_x$ and $x$ such that   $I_x=(\beta_x, x]_X$ has the same cardinality as any smaller neighborhood of $x$.  Then $D=X\setminus \cup \{I_x:x\in X'\}$ is a clopen discrete subset of $X$ and $X=\left ( \oplus_{x\in X'}I_x\right ) \oplus D$ is a desired representation.
\end{proof}

\par\bigskip\noindent
To prove our target statement, first for each infinite cardinal $\gamma$, we identify a   linearly ordered topological space $\langle L_\gamma, \prec\rangle$ for which $\langle L_\gamma, \prec\rangle\times \langle L_\gamma, \prec\rangle$ is homeomorphic to $\langle L_\gamma, \prec\rangle\times_l \langle L_\gamma, \prec\rangle$. Next, we will direct our efforts on the task of embedding the members of $\mathcal L$  into such spaces. 

\par\bigskip\noindent
{\bf Construction  of $\langle L_\gamma, \prec\rangle$ for an infinite cardinal $\gamma$}. 
\par\smallskip\noindent
\underline {\it Definition of $L_\gamma$.} Denote by $\lambda_\gamma$ the ordinal $(\gamma\times_l \gamma ) +1$. Define $L_\gamma$ as the  subspace of $\lambda_\gamma$ that consists of all points $\alpha$ 
that fall into one of the following three categories:
\begin{enumerate}
	\item 	$\alpha=\max \lambda_\gamma$
	\item  $[\alpha_0, \alpha]$ is order-isomorphic to $\gamma + 1$ for some $\alpha_0<\alpha$.
	\item  $\alpha$ is isolated.
\end{enumerate}
{\it Remark.} To help visualize $L_\gamma$, put $I=\{\alpha<\gamma:\alpha \ is \ isolated\}\cup \{\gamma\}$. Then $L_\gamma$ can be thought of as a $\gamma$-long sequence of $\gamma$-many clopen copies of $I$ converging to $\max \lambda_\gamma$. 

\par\medskip\noindent
\underline {\it Definition of $\prec$.} If $\gamma = \omega$, then $L_\gamma = \lambda_\gamma$ and we let $\prec$ be equal to the existing ordering $<$. For $\gamma>\omega$, we will define $\prec$ using a folklore ordering procedure. We first define the order formally and then follow up with a simple demonstration. For each $\alpha\in \lambda_\gamma\setminus L_\gamma$,  put $R_\alpha = \{\alpha + 1, \alpha + 2, ...\}$. By the definition of $L_\gamma$ and the fact that $\gamma>\omega$, we conclude that $R_\alpha$ is a closed subset of  $L_\gamma$ . Define  $\prec_\alpha$  on $R_\alpha$ as follows: $...\alpha + 5\prec \alpha + 3 \prec \alpha + 1 < \alpha + 2 \prec \alpha + 4 ...$. Define $\prec$ as follows:
\begin{enumerate}
	\item $x\prec y$ if   $x,y\in R_\alpha$ and $x\prec_\alpha y$.
	\item $x\prec y$ if $\{x,y\}$ is not a subset of $R_\alpha$ for any $\alpha$ and $x<y$.
\end{enumerate}
Construction of $\langle L_\gamma,\prec\rangle$ is complete.
\begin{flushright}$\square$\end{flushright}

\par\bigskip\noindent
To convince a reader that the above definition is legal without going into painful details,  let us demonstrate  a folklore construction of a topology-compatible order for the space $X=\{-1/n: n=1,2,3,...\} \cup \{5-1/n: n = 1,2,3,...\}$. The space $X$ is not a linearly ordered space but there are many simple topology-compatible orders on $X$. The one that mimics the above construction is defined as follows. First, reverse the order on $\{5-1/n:n=1,2,3...\}$. The resulting set becomes order isomorphic to $\{\pm 1/n:n=1,2,3,..\}$ and is homeomorphic to $X$. This short construction is formalized in the above definition  in which we top every  "missing limit point" by the reversed  sequence "converging to the next missing limit point".

\par\bigskip\noindent
Note that in our definition of $\prec$ for $L_\gamma$ we do not change the order position of limit points of $L_\gamma$, which means that the new order coincides with the natural order when one of the compared elements is in  $L_\gamma'$. In a sense, the new order $\prec$ on $L_\gamma$ is almost indistinguishable from the standard order $<$ if "{\it observed  from far away}".  Also note that if $X$ is  a subspace of an ordinal that is homogeneous on the derived set,  then
by Lemma \ref{lem:freesum}, $X$ can be embedded into $L_\gamma$ for some $\gamma$.   Let us record these observations for future reference.

\par\bigskip\noindent
\begin{lem}\label{lem:embed} The following hold:
\begin{enumerate}
	\item Every $X\in \mathcal L$ embeds in $\L_\gamma$ for some cardinal $\gamma$.
	\item If $x\in L_\gamma'$, $y\in L_\gamma$, and $x<y$, then $x\prec y$.
	\item If $x\in L_\gamma'$, $y\in L_\gamma$, and $y<x$, then $y\prec x$.
\end{enumerate}
\end{lem}

\par\bigskip\noindent
We will often use the facts in this summary lemma without explicit referencing. Our next goal is to show that lexicographical and Cartesian product operations produce topologically equivalent results when applied to an $L_\gamma$. We start by considering the two operations on smaller pieces of $L_\gamma$'s. In the following three statements the arguments will be very similar to each other. For clarity, we will also use similar wording.

\par\bigskip\noindent
\begin{lem}\label{lem:zerogammasquared}
Let $\gamma$ be an infinite cardinal. Then  $[0,\gamma]_{ L_\gamma}\times [0,\gamma]_ {L_\gamma}$ is homeomorphic to $L_\gamma$.
\end{lem}
\begin{proof}
To prove the statement we will visualize $[0,\gamma]_{L_\gamma}$ and $L_\gamma$ as described in the remark after the definition of $L_\gamma$. Namely,  $[0,\gamma]_{L_\gamma}=I=\{\alpha<\gamma:\alpha\ is\  isolated\}\cup \{\gamma\}$ and $L_\gamma$ is  a $\gamma$-long sequence of $\gamma$-many clopen copies of $I$ converging to $\infty=\max L_\gamma$.  We can write then $L_\gamma=(\oplus \{I_\alpha=I: \alpha<\gamma,\alpha\ is\ isolated\})\cup \{\infty\}$, where every neighborhood of $\infty$ contains all $I_\alpha$'s starting from some moment. Having these visuals in mind we will construct a desired homeomorphism in three stages as follows:
\begin{description}
	\item[\it Stage 1] Partition the set of isolated ordinals of $\gamma$ into pairs $\{\{a_\alpha,b_\alpha\}:\alpha<\gamma,\alpha\ is\ isolated\}$ so that $b_\alpha = a_\alpha + 1$ and indexing  agrees with the natural well-ordering $<$ of the partitioned set. 
	\item[\it Stage 2] Since $\gamma$ is an infinite cardinal, $[\alpha,\gamma]_{L_\gamma}$ is homeomorphic to $I$ for any $\alpha<\gamma$. Therefore, for each isolated $\alpha<\gamma$ we can fix homeomorphisms 
$g_\alpha:\{\alpha\}\times [\alpha,\gamma]_{L_\gamma}\to I_{b_\alpha}$ and $h_\alpha:(\alpha,\gamma]_{L_\gamma}\times \{\alpha\}\to I_{a_\alpha}$. That is, $g_\alpha$ maps the $\alpha$'s vertical of $[0,\gamma]_{ L_\gamma}\times [0,\gamma]_ {L_\gamma}$ at or above the diagonal onto $b_\alpha$'s copy of $I$ in $L_\gamma$ and $h_\alpha$ maps the $\alpha$'s horizontal strictly below the diagonal onto $a_\alpha$'s copy of $I$.
	\item[\it Stage 3] Define a homomorphism $f$ from  $[0,\gamma]_{ L_\gamma}\times [0,\gamma]_ {L_\gamma}$ to $L_\gamma$ as follows:

$$
f(p) = \left\{
        \begin{array}{lll}
             g_\alpha(p) & if & p \in  \{\alpha\}\times [\alpha,\gamma]_{L_\gamma}\\
             h_\alpha(p)& if & p\in (\alpha,\gamma]_{L_\gamma}\times \{\alpha\} \\
	  \infty & if & p=\langle \gamma,\gamma\rangle
        \end{array}
    \right.
$$
\end{description}
The argument similar to that in Example \ref{ex:convseq} shows that $f$ is a homeomorphism.
\end{proof}

\par\bigskip\noindent
\begin{lem}\label{lem:zerogammatimesall}
Let $\gamma$ be an infinite cardinal. Then  $[0,\gamma]_{   L_\gamma}\times  L_\gamma$ is homeomorphic to $[0,\gamma]_{  \langle L_\gamma, \prec\rangle}\times_l  \langle L_\gamma, \prec\rangle$.
\end{lem}
\begin{proof}
Denote the spaces in the statement by $X$ and $Y$, respectively. Since $Y$ is homeomorphic to $Z=Y\setminus (\{\gamma\}\times_l [1,\gamma]_{  \langle L_\gamma, \prec\rangle})$, it suffices to  construct an isomorphism from $X$ to $Z$, which we will do next.

When treating  $[0,\gamma]_{  \langle L_\gamma, \prec\rangle}$ and $  \langle L_\gamma, \prec\rangle$ as topological spaces with regard to order,  we will visualize them as described in  Lemma \ref{lem:zerogammasquared}. For convenience, let us copy our notation from Lemma \ref{lem:zerogammasquared} next:

 $$[0,\gamma]_{L_\gamma}=I=\{\alpha<\gamma:\alpha\ is\  isolated\}\cup \{\gamma\}$$
$$L_\gamma=(\oplus \{I_\alpha=I: \alpha<\gamma,\alpha\ is\ isolated\})\cup \{\infty\},$$
where $\infty$ is the maximum element of $L_\gamma$ in either of the two orders.
We are now ready to construct a desired homeomorphism in three stages as follows:
\begin{description}
	\item[\it Stage 1] Partition the set of isolated ordinals of $\gamma$ into pairs $\{\{a_\alpha,b_\alpha\}:\alpha<\gamma,\alpha\ is\ isolated\}$ so that $b_\alpha = a_\alpha + 1$ and indexing  agrees with the natural well-ordering $<$ of the partitioned set. 
	\item[\it Stage 2] By Lemma \ref{lem:zerogammasquared}, for each isolated ordinal $\alpha<\gamma$ there exists a homeomorphism $h_\alpha$ of $[\alpha,\gamma]_{ L_\gamma}\times I_\alpha$ onto $\{a_\alpha\}\times_l \langle L_\gamma, \prec\rangle$. Since $\gamma$ is an infinite cardinal, $V_\alpha= L_\gamma\setminus \bigcup_{\beta\leq \alpha} I_\beta$ is homeomorphic to $L_\gamma$. Hence, we can find a homeomorphism $g_\alpha$ from 
$\{\alpha\}\times V_\alpha$ onto $\{b_\alpha\}\times_l \langle L_\gamma, \prec\rangle$.
	\item[\it Stage 3] Define a homomorphism $f$ from  $X$ to $Z$ as follows:

$$
f(p) = \left\{
        \begin{array}{lll}
             g_\alpha(p) & if & p \in  \{\alpha\}\times V_\alpha\\
             h_\alpha(p)& if & p\in [\alpha,\gamma]_{ L_\gamma}\times I_\alpha\\
	  \langle \gamma, \infty\rangle & if & p=\langle \gamma,\infty\rangle
        \end{array}
    \right.
$$
\end{description}
In words, $f$ maps most of the $\alpha$'s horizontal strip corresponding to $I_\alpha$ onto the $a_\alpha$'s copy of $\langle L_\gamma,\prec\rangle$, most of the $\alpha$'s vertical onto $b_\alpha$'s copy of  $\langle L_\gamma,\prec\rangle$, and the corner point of the Cartesian product to the maximum of $Z$.  The argument similar to that in Example \ref{ex:convseq} shows that $f$ is a homeomorphism.
\end{proof}

\par\bigskip\noindent
We are now ready to prove a generalization of the statement of Example \ref{ex:seqofsequences}.
\par\bigskip\noindent
\begin{lem}\label{lem:main} For every infinite cardinal $\gamma$, the space $L_\gamma\times  L_\gamma$ is homeomorphic to $\langle L_\gamma, \prec\rangle\times_l \langle L_\gamma, \prec\rangle$.
\end{lem}
\begin{proof} Denote by $X$ and $Y$ the two spaces in the statement. As in Lemma \ref{lem:zerogammatimesall}, it suffices to construct a homeomorphism from $X$ to $Z= Y\setminus (\{\infty\}\times_l [1, \infty]_{\langle L_\gamma, \prec\rangle})$
As in the previous two lemmas, we visualize $[0,\gamma]_{L_\gamma}$ and $L_\gamma$ as follows:

$$[0,\gamma]_{L_\gamma}=I=\{\alpha<\gamma:\alpha\ is\  isolated\}\cup \{\gamma\}$$
$$L_\gamma=(\oplus \{I_\alpha=I: \alpha<\gamma,\alpha\ is\ isolated\})\cup \{\infty\},$$
where $\infty$ is the maximum element of $L_\gamma$ in either of the two orders.
We will closely follow our constructions in the previous two lemmas and construct the promised  homeomorphism in three stages as follows:

\begin{description}
	\item[\it Stage 1] Partition the set of isolated ordinals of $\gamma$ into pairs $\{\{a_\alpha,b_\alpha\}:\alpha<\gamma,\alpha\ is\ isolated\}$ so that $b_\alpha = a_\alpha + 1$ and indexing  agrees with the natural well-ordering $<$ of the partitioned set. 
	\item[\it Stage 2] By Lemma \ref{lem:zerogammatimesall}, for each isolated $\alpha<\gamma$, we can fix two homeomorphisms:
$$
h_\alpha:L_\gamma\times I_\alpha\to \langle I_{a_\alpha},\prec\rangle\times_l\langle L_\gamma,\prec\rangle
$$

$$
g_\alpha : I_\alpha\times \left (L_\gamma\setminus \bigcup_{\beta\leq \alpha} I_\beta\right )\to \langle I_{b_\alpha},\prec\rangle\times_l\langle L_\gamma,\prec\rangle
$$
	\item[\it Stage 3]  Define a homomorphism $f$ from  $X$ to $Z$ as follows:

$$
f(p) = \left\{
        \begin{array}{lll}
             g_\alpha(p) & if & p \in   I_\alpha\times \left (L_\gamma\setminus \bigcup_{\beta\leq \alpha} I_\beta\right )\\
             h_\alpha(p)& if & p\in L_\gamma\times I_\alpha\\
	  \langle \infty, \infty\rangle & if & p=\langle \infty,\infty\rangle
        \end{array}
    \right.
$$
\end{description}
In words, $f$ maps most of the $\alpha$'s horizontal strip onto the $a_\alpha$'s copy of $I\times_l L_\gamma$, most of  the $\alpha$'s vertical strip onto  the $b_\alpha$'s copy of $I\times_l L_\gamma$, and the corner point of the Cartesian product to the maximum of $Z$. An argument similar to one of Example \ref{ex:convseq} shows that $f$ is a homeomorphism.
\end{proof}

\par\bigskip\noindent
Lemmas \ref{lem:main} and \ref{lem:embed} imply the following  main statement of our discussion.
\begin{thm}\label{thm:main}
Let $X$ be a subspace of an ordinal that is homogeneous on the derived set. Then $X$ can be embedded into a LOT $Z$  such that $Z\times_l Z$ is homeomorphic to $Z\times Z$.
\end{thm}

\par\bigskip
In search for candidates with the discussed phenomenon, it is clear that we should immediately eliminate any ordered spaces with stationary subsets. Indeed, the square of such a space is not orderable as follows from a standard generalization of Katetov's example \cite{Kat}. Therefore, by the characterization of hereditary paracompactness for GO-spaces due to Engelking and Lutzer (\cite{BL} or \cite{Lut}), we should consider only hereditary paracompact ordered spaces.  It is clear that  if $X$ has no stationary subset, then  $X^2$ does not have such either. Thus, we need to concentrate on spaces with orderable hereditary paracompact squares. While the Engelking-Lutzer characterization is incredibly handy for  testing an ordered space for hereditary paracompactness, the author is not aware of any criterion for the square of a LOTS to be hereditary paracompact. Is there such a criterion? If not, let us find one!

\par

\end{document}